\documentclass[12pt]{article}

\usepackage[margin=1.5in]{geometry}
\usepackage{amsmath,amsfonts,graphicx,framed}
\usepackage{amssymb}
\usepackage{amsthm}
\usepackage{fancyhdr}
\usepackage{float}
\usepackage{caption}
\usepackage{comment}
\usepackage[utf8]{inputenc}
\usepackage{authblk}
\usepackage[mathscr]{euscript}
\usepackage{mathtools}
\usepackage{xcolor}

\usepackage[hypertexnames=false]{hyperref}

\theoremstyle{plain}
\newtheorem{thm}{Theorem}
\newtheorem{lem}{Lemma}
\newtheorem{cor}{Corollary}

\newtheorem{prop}{Proposition}
\theoremstyle{definition}
\newtheorem{defn}{Definition}
\newtheorem{rem}{Remark}

\newcommand{\pd}{\partial}

\newcommand{\mbb}{\mathbb}

\newcommand{\ep}{\varepsilon}
\newcommand{\re}{\mbb R}
\newcommand{\al}{\alpha}
\newcommand{\mc}{\mathcal}

\newcommand{\eqal}[1]{\begin{equation}\begin{aligned}#1\end{aligned}\end{equation}}

\newcommand{\ov}{\overline}

\newcommand{\ti}{\tilde}

\begin{document}

\title{Recovering a quasilinear conductivity from boundary measurements}

\author{Ravi Shankar}

\date{\today}

\maketitle

\begin{abstract}
We consider the inverse problem of recovering an isotropic quasilinear conductivity from the Dirichlet-to-Neumann map when the conductivity depends on the solution and its gradient.  We show that the conductivity can be recovered on an open subset of small gradients, hence extending a partial 
result 
to all real analytic conductivities.  
We also recover non-analytic conductivities with additional growth assumptions along large gradients.  Moreover, the results hold for non-homogeneous conductivities if the non-homogeneous part is assumed known.
\end{abstract}

\section{Introduction}

\subsection{Statement of results}
Let $\Omega\subset\re^n$, $n\ge 2$, be a smooth bounded domain, and $u(x)\in C^\infty(\ov\Omega)$ be the smooth solution of the quasilinear boundary value problem
\eqal{\label{eq}
\nabla\cdot(a(u,\nabla u)\nabla u)&=0,\qquad x\in\Omega,\\
u|_{\pd\Omega}&=f,
}
where $a(s,p)\in C^\infty(\re\times\re^n)$ satisfies standard conditions to make this BVP wellposed, see \cite{MU} and \cite{GT}.  We assume summation over repeated indices.

\medskip
\noindent
\textit{Ellipticity:}
\eqal{\label{aij}
0<\lambda(s,p)|\xi|^2\le &\,a_{ij}(s,p)\xi_i\xi_j,\quad (s,p,\xi)\in\re\times\re^n\times(\re^n\setminus\{0\}),\\
&a_{ij}:=a\,\delta_{ij}+\frac{1}{2}(a_{p_i}p_j+a_{p_j}p_i).
}
\textit{Coercivity:}
\eqal{
\label{coer}
\lambda(s,p)\ge \lambda_0(|s|)>0,\\
a(s,p) \ge 1.
}
\textit{Growth:}
\eqal{
\label{grow}
|p||\nabla_p a|+|a|\le \mu_0(|s|),\\
(1+|p|)|a_s|\le\mu_0(|s|)|p|.
}
The bounds $\lambda,\mu_0,\lambda_0
$ and domain $\Omega$ are known quantities.  Positive functions $\lambda_0$ and $\mu_0$ are non-increasing and non-decreasing in $|s|$, respectively. 
We also assume known bounds on the $C^k(K)$ seminorms of $a$ for all $k\ge 0$ and all compact subsets $K$ of the domain of $a$.

\medskip
Now, the Dirichlet-to-Neumann (DN) map is defined by
\eqal{\label{DN}
f\mapsto \Gamma_a(f) 
&=a(u,\nabla u)\frac{\pd u}{\pd \nu}dS,
}
where $dS$ is the Euclidean area form on $\pd\Omega$, and $\frac{\pd u}{\pd \nu}=\nu\cdot\nabla u$ is the Euclidean normal derivative.  

\medskip
Under no additional assumptions, we can recover $a$ from $\Gamma_a$ along sufficiently small gradient values.

\begin{thm}\label{thm:open}
If $\Gamma_a=\Gamma_{\ti a}$, then $a(s,p)=\ti a(s,p)$ on an open set:
$$
\{(s,p):|p|< \Pi(s)\},
$$ 
where $\Pi:\re\to(0,\infty)$ is a positive continuous function which depends on the known quantities.
\end{thm}

\medskip
We next recover two classes of $a(s,p)$ along large gradients.  The first follows immediately from Theorem \ref{thm:open}.

\begin{thm}\label{thm:analytic}
Let $a(s,p)$ and $\ti a(s,p)$ be real-analytic.  If $\Gamma_a=\Gamma_{\ti a}$, then $a(s,p)\equiv \ti a(s,p)$.
\end{thm}

Theorem \ref{thm:analytic} extends the result of \cite{MU}, which assumed conditions on the holomorphic extension of real analytic $a(s,p)$ in order to solve a complex-valued BVP.  

\begin{thm}\label{thm:decay}
Let $a$ and $\ti a$ satisfy two additional conditions:

\medskip
1. Uniform elliptic lower bound:
\eqal{
\label{uni}
a_{ij}(s,p),\quad \ti a_{ij}(s,p)\ge \lambda_0>0,\qquad (s,p)\in\re\times\re^n,
}
where the inequality is in the quadratic form sense, and $\lambda_0$ is constant.

\medskip
2. Decay along large gradients:
\eqal{
\label{decay}
|a_s(s,p)|,\quad |\ti a_s(s,p)|\le  \frac{\ep_0}{|p|},\qquad (s,p)\in\re\times\re^n,
}
for some sufficiently small $\ep_0>0$ depending on $\Omega$ and $\lambda_0$.

\medskip
If $\Gamma_a=\Gamma_{\ti a}$, then $a(s,p)\equiv \ti a(s,p)$.
\end{thm}

This seems to be the first example of global uniqueness of a non-analytic conductivity depending on $(u,\nabla u)$.  

\medskip
Finally, as shown in Remarks \ref{rem:n3}, \ref{rem:bsym}, and \ref{rem:n2}, the above results hold for non-homogeneous conductivities $a=a(x,s,p)$ as well, provided the ``$x$-dependent part" is either known  (i.e. is the same for both $a$ and $\ti a$) or symmetric along a direction, say $(0,\dots,0,1)\in\re^n$.  
We assume the known quantities in \eqref{aij}, \eqref{coer}, and \eqref{grow} 
 are unchanged and do not depend on $x$, for simplicity.

\begin{cor}\label{cor:x}
If $n\ge 3$, suppose also $\Omega$ is convex and $a$ and $\ti a$ are of the form
\eqal{
a(x,s,p)=F(x,s,p,b(x^1,\dots,x^{n-1},s,p)),\\
\ti a(x,s,p)=F(x,s,p,\ti b(x^1,\dots,x^{n-1},s,p)),
}
where $t\mapsto F(x,s,p,t)$ is injective for each fixed $(x,s,p)$.  (If $n=2$, remove convexity of $\Omega$ and assume $b$ and $\ti b$ are functions of $(s,p)$ only).

\medskip
Then the statements of Theorems \ref{thm:open}, \ref{thm:analytic}, and \ref{thm:decay} are true with replacements of $(s,p)$ by $(x,s,p)$, with $x\in\Omega$.
\end{cor}

The fully general case where $b=b(x,s,p)$ is unclear to the author and may require new methods. 

\subsection{Discussion}
The classical Calder\'on problem considers an isotropic linear elliptic equation and its Dirichlet-to-Neumann map
\eqal{
\label{cald}
\nabla\cdot (a(x)\nabla u)=0,\\
\Gamma_a:u|_{\pd\Omega}\mapsto a(x)\frac{\pd u}{\pd \nu}\Big|_{\pd\Omega}.
}
The problem is to determine $a(x)$ from $\Gamma_a$.  See \cite{U} for a survey of this problem and its extensions.  The matrix extension $a(x)=a_{ij}(x)$, or the anisotropic Calder\'on problem, is still unsolved in general for dimension three or higher.

\medskip
Recent work has featured the quasilinear generalization $a=a(x,u,\nabla u)$.  The case $a=a(x,u)$ was considered in \cite{S1}, using a linearization technique similar to that originally introduced in \cite{I1}.  This was extended in \cite{SU} to the anisotropic case $a=a_{ij}(x,u)$ (matrix).  See \cite{HS} for analysis in the case $a=a_{ij}(x,\nabla u)$ in dimension two, and \cite{KN} for the recovery of a Taylor polynomial of $a=a_{ij}(x,\nabla u)$.  The linearization technique was also applied to the semilinear equations $\Delta u+a(x,u)=0$ and $\Delta u+a(x,\nabla u)=0$ by \cite{IS} and \cite{S2}, respectively.  More recently, a multilinearization technique was developed for $\Delta_gu+a(x,u)=0$ in \cite{FO,KU,LLLS}, where $a$ is analytic in $u$.  The equation $\Delta u+f(x)|\nabla u|^2+V(x,u)=0$ with $V$ analytic was determined in the partial data case in \cite{KU1} using multilinearization.  
In \cite{I2}, the equation $\Delta u+c(u,\nabla u)=0$ was considered using singular solutions of the linearized equation which concentrate at the boundary.  The gradient structure was used for the quasilinear equation $\nabla\cdot(a(u)\nabla u)+c(x)u=0$ with lower order term in \cite{EPS}.  

\medskip
In the work \cite{MU} by Mu\~noz and Uhlmann, attention was brought to the quasilinear case $a=a(u,\nabla u)$, and results were obtained assuming conditions on the holomorphic extension of $a$.  These conditions ensure that linearization is possible around complex-valued affine solutions.  The question is raised whether we can allow for \textit{arbitrary} $a(s,p)$ real analytic, $a(s,p)$ smooth but \textit{non-analytic}, and \textit{non-homogeneous} $\pd a(x,s,p)/\pd x\neq 0$.

\medskip
The purpose of this work is to extend the results of \cite{MU} to several new classes of quasilinear conductivities:

\medskip
1. $a(u,\nabla u)$ real analytic, without additional hypotheses, Theorems \ref{thm:open} and \ref{thm:analytic}.

\medskip
2. $a(u,\nabla u)$ smooth, assuming additional conditions \eqref{uni} and \eqref{decay}, Theorem \ref{thm:decay}.

\medskip
3. Non-homogeneous cases $a=a(x,u,\nabla u)$ with known $x$ part, Corollary \ref{cor:x}.

\medskip
An apparent \textit{a priori} difficulty of the inverse problem for $a=a(u,\nabla u)$ (and truly one for general $a=a(x,u,\nabla u)$) is that the linearized inverse problem is anisotropic.  Not only is the anisotropic problem not understood in dimension three or higher, but the matrix $a_{ij}(x)$ can only be recovered up to pushforward by a boundary-preserving diffeomorphism; see \cite{SU} and \cite{S2} for similar difficulties.  In the $p$-Laplace case $a=A(x)|\nabla u|^{p-2}\nabla u$, the ``nonlinear part" of the quasilinearity is known, unlike for equation \eqref{eq}, and boundary determination was possible in \cite{SX} using explicit Wolff-type oscillatory solutions, with full recovery using a monotonicity identity \cite{GKS}.  If instead the anisotropic and nonlinear part is small, then one may linearize near the zero solution as in \cite{KN} and recover the $x$-dependent Taylor coefficients using linear theory.  This is similar to the idea of the multilinearization approach \cite{FO,KU,KU1,LLLS} for which dependence on $u$ or $\nabla u$ is analytic.  Constant solutions were used for linearizing the conductivity equation in \cite{SU}, and complex-valued affine solutions were used for linearization in \cite{MU}.  
However, it is not clear what non-constant, explicit solutions quasilinear equation \eqref{eq} has in full generality.  If $a(s,p)$ is smooth but not real analytic, then complex-valued affine solutions no longer make sense.

\medskip
Our approach is to use classical boundary determination theory, which is well understood even in the anisotropic case.  We first linearize the equation around any solution with a prescribed boundary jet, then recover the tangential part of the linearized conductivity matrix \eqref{aij} using boundary determination theory.  Next, the isotropic structure of the linearized conductivity matrix yields a solvable algebraic system for the conductivity evaluated at the jet.  
According to the comparison principle, we can use logarithmic barriers to prescribe boundary jets on a small open subset, which will complete the proof of Theorem \ref{thm:open}.   For Theorem \ref{thm:decay}, decay condition \eqref{decay} ensures that exponential barriers exist, hence solutions with any prescribed boundary jet in $\re\times\re^n$.  

\medskip
This work is organized as follows.  In Section \ref{sec:lin}, we show that the linearized BVP is wellposed and has a well-defined DN map.  In Section \ref{sec:jet}, we use the comparison principle to construct solutions with prescribed boundary jets.  In Section \ref{sec:geo}, we geometrically reformulate the linearized problem, which allows us to invoke known boundary determination theory in Section \ref{sec:bdy}.  In Section \ref{sec:proof}, we complete the proofs of Theorems \ref{thm:open} and \ref{thm:decay}.

\section{Preliminaries}


\subsection{The linearized problem}
\label{sec:lin}

For arbitrary solutions to the quasilinear equation \eqref{eq}, the linearized Dirichlet problem centered at that solution may not be solvable.  Under a small gradient assumption or the decay assumption \eqref{decay}, we will show the uniqueness of weak solutions to the linearized equation, and standard techniques will again apply after invoking higher regularity.  Note that the non-homogeneous case $a=a(x,u,\nabla u)$ applies here without change as well.

\medskip
We denote by $u=u[f]$ the solution of the BVP \eqref{eq}.  If $f,h\in C^\infty(\pd\Omega)$, then formally, we let 
\eqal{
\label{vdef}
v(x):=\frac{d}{dt}\Big|_{t=0}u[f+t h](x)=\lim_{t\to 0}\frac{u[f+th](x)-u[f](x)}{t}
}
solve the linearized Dirichlet problem
\eqal{\label{eqL}
\pd_i(a_svu_i+a_{p_j}v_ju_i+av_i)&=0,\\
v|_{\pd\Omega}=h.
}
The linearized DN map:
\eqal{\label{DNL}
\Gamma_a[u](h):=\frac{d}{dt}\Big|_{t=0}\Gamma_a(f+t h)=(a_sv\frac{\pd u}{\pd \nu}+a_{p_j}v_j\frac{\pd u}{\pd \nu}+a\frac{\pd v}{\pd \nu})dS.
}

\medskip
For brevity, we say that \textit{the linearization of \eqref{eq} exists at $u$} if the limit \eqref{vdef} is well defined in $C^{1}(\Omega)$ and the Dirichlet problem \eqref{eqL} for $v$ is uniquely solvable in $W^{1,2}(\Omega)$.  If $\Gamma_a=\Gamma_{\ti a}$ and the linearization exists at $u$ and $\ti u$, it then follows that $\Gamma_a[u]=\Gamma_{\ti a}[\ti u]$.

\medskip
To show that the linearization of \eqref{eq} exists, we first specialize to solutions $u$ with small gradients, under no additional assumptions on $a(s,p)$.

\begin{prop}
\label{prop:lin_small}
Given $s\in\re$, there exists $\Pi_1>0$ sufficiently small such that if $f\in C^\infty(\pd\Omega)$ satisfies
$$
\|f-s\|_{C^{2,\al}(\pd\Omega)}\le \Pi_1,
$$ 
then the linearization of \eqref{eq} exists at $u=u[f]$.  Here, $\al\in(0,1)$ is as in \cite[Theorem 2.10]{MU} and depends on the known quantities and $s$.
\end{prop}

To linearize \textit{near} constant solutions, we need to verify the solution operator is Lipschitz continuous \textit{at} constants.
\begin{lem}
Given $k\in\re$ and $h\in C^{2,\al}(\pd\Omega)$, there exists $\al\in(0,1)$ sufficiently small such that the following estimate holds:
\eqal{
\label{C2ak}
\|u[k+h]-k\|_{C^{2,\al}}\le C(k,\|h\|_{C^2})\|h\|_{C^{2,\al}}.
}
\end{lem}

\begin{proof}
Letting $f=k+h$ and $u=k+w$ in \eqref{eq}, we obtain the BVP
\eqal{
\label{eqt}
\nabla\cdot\left(a(k+w,\nabla w)\nabla w\right)=0,\\
w|_{\pd\Omega}=h.
}
The structural conditions \eqref{aij}, \eqref{coer}, and \eqref{grow} are still satisfied, with lower and upper bounds now also depending on a parameter $k$.  It was shown in \cite{MU} that \eqref{eq} has an \textit{a priori} $C^{1,\al}$ estimate in terms of $\|f\|_{C^2}$ for some $\al\in(0,1)$ depending on $a$ and $\|f\|_{C^2}$, which means each coefficient in the nondivergence form of \eqref{eqt},
\eqal{
a_{ij}(k+w,\nabla w)w_{ij}+(a_w(k+w,\nabla w)w_i)w_i=0,
}
has a $C^\al$ estimate also in terms of $\|f\|_{C^2}$.  Therefore, the linear Schauder $C^{2,\al}$ estimate \cite[Theorem 6.6]{GT} yields
\eqal{
\label{C2a}
\|u\|_{C^{2,\al}(\Omega)}\le C(a,\Omega,\|f\|_{C^2(\pd\Omega)})\|f\|_{C^{2,\al}(\pd\Omega)}.
}
In the case of \eqref{eqt}, the same estimate is therefore true, with $C(a,\Omega,\|f\|)$ replaced by $C(k,a,\Omega,\|h\|)=:C(k,\|h\|)$, and $\al$ depending also on $k$ and $\|h\|_{C^2}$.  Substitution of $w=u-k$ into the above completes the proof.
\end{proof}

We also need the continuity of the solution operator near \textit{any} small solution, which we record for completeness.  Given $M>0$ and the sufficiently small $\al=\al(M)$ of \eqref{C2a}, we give the subset $B_M=\{f\in C^{2,\al}(\pd\Omega):\|f\|_{C^2(\pd\Omega)}< M\}$ the $C^{2,\al}(\pd\Omega)$ norm.

\begin{lem}
\label{lem:cont}
For each $M>0$ and $\beta\in(0,\al(M))$, the operator $B_M\to C^{2,\beta}(\Omega)$ given by $f\mapsto u[f]$ is continuous.
\end{lem}

\begin{proof}
If not, then there exists a sequence $f_k\in C^{2,\al}(\pd\Omega)$ with $\|f_k\|_{C^2}\le M$ converging to $f\in C^{2,\al}(\pd\Omega)$ but with $\|u[f_k]-u[f]\|_{C^{2,\beta}(\Omega)}\ge 1$, say.  By \eqref{C2a}, the solutions $u[f_k]$ are uniformly bounded in $C^{2,\al}(\ov\Omega)$, so the Arzela-Ascoli theorem yields a subsequence, also named $u[f_k]$, which converges in $C^{2,\beta}(\ov\Omega)$ to some $v$.  It follows that $\|v-u[f]\|_{C^{2,\beta}(\Omega)}\ge 1$.  However, since each $u[f_k]$ solves \eqref{eq} with boundary data $f_k$, sending $k\to\infty$ in \eqref{eq} shows that $v$ solves \eqref{eq} as well with boundary data $f$.  By the comparison principle \cite[Theorem 10.7]{GT} for quasilinear divergence form operators, it follows that $C^{2,\beta}$ solutions are unique, hence $v=u[f]$, a contradiction.
\end{proof}

\begin{proof}[Proof of Proposition \ref{prop:lin_small}]
Defining $t$-dependent quantities,
\eqal{
u(x;t):=u[f+th](x),\quad v(x;t):=\frac{u(x;t)-u(x;0)}{t},\\
a(x;t):=a[u(x;t)],\quad u(;\sigma;t)=u(;0)+\sigma(u(;t)-u(;0)),
}
with $a[u]:=a(u,\nabla u)$, we take a ``discrete derivative":
\eqal{
a(;t)u_i(;t)-a(;0)u_i(;0)&=ta(;t)v_i(;t)+(a(;t)-a(;0))u_i(;0)\\
&=tav_i+ta_s(;t)v(;t)u_i+ta_{p_j}(;t)v_j(;t)u_i,
}
where, by the fundamental theorem of calculus,
\eqal{
a_s(;t):=\int_0^1a_s[u(;\sigma;t)]d\sigma,\quad a_{p_j}(;t):=\int_0^1a_{p_j}[u(;\sigma;t)]d\sigma.
}
We see that $v(;t)$ then solves the Dirichlet problem
\eqal{
\label{probt}
\pd_i\left(a_s(;t)v(;t)u_i(;0)+a_{p_j}(;t)v_j(;t)u_i(;0)+a(;t)v_i(;t)\right)=0,\\
v(;t)|_{\pd\Omega}=h,
}
with associated Cauchy data (DN map)
\eqal{
\left(a_s(;t)v(;t)\frac{\pd u}{\pd \nu}(;0)+a_{p_j}(;t)v_j(;t)\frac{\pd u}{\pd\nu}(;0)+a(;t)\frac{\pd v}{\pd\nu}(;t)\right)dS.
}

\medskip
\textbf{Step 1:}

We verify the Dirichlet problem \eqref{probt} for $v(;t)$ has a unique weak solution in $W^{1,2}(\Omega)$ which is uniformly bounded in $t$.  Let $w\in W^{1,2}_0$ solve
\eqal{
\label{eqLw}
\pd_i\left(a_s(;t)wu_i(;0)+a_{p_j}(;t)w_ju_i(;0)+a(;t)w_i\right)=0.
}
Multiplying \eqref{eqLw} by $w$ and integrating by parts, we obtain
\eqal{
\int_\Omega a_{ij}(;t)w_iw_j\,dx=-\int_\Omega a_s(;t)u_i(;0)ww_i\,dx,
}
where we defined
\eqal{
a_{ij}(;t):=a(;t)\delta_{ij}+\frac{1}{2}\int_0^1 u_i(;0)a_{p_j}[u(;\sigma;t)]+u_j(;0)a_{p_i}[u(;\sigma;t)]\,d\sigma.
}
For $|t|\le 1$, estimate \eqref{C2a} shows $u(;t)$ is uniformly bounded in $C^{2,\al}$ with respect to $t$, for $\al$ sufficiently small.  Moreover, Lemma \ref{lem:cont} shows that $t\mapsto u(;t)$ is continuous in $C^{2,\beta}$ if $\beta\in(0,\al)$, so by the smoothness of $a$, we can write $a_{ij}(;t)=a_{ij}(;0)+r_1(t)$, where remainder $r_1(t)$ vanishes in $C^{1,\beta}$ as $t\to 0$.  By \eqref{grow}, we can therefore suppose $t$ is sufficiently small so that
\eqal{\label{aijbound}
a_{ij}(;t)\xi_i\xi_j\ge \frac{1}{2}\lambda_0(|u(;0)|)|\xi|^2,\qquad \xi\in\re^n.
}
Similarly, we write $a_s(;t)=a_s(;0)+r_2(t)$, and recalling \eqref{grow}, we suppose $t$ is sufficiently small so that
\eqal{\label{asbound}
|a_s(;t)|\le 2\mu_0(|u(;0)|).
}
The Cauchy-Schwarz inequality then yields
\eqal{
\label{ineq4}
\int_\Omega\lambda_0(|u(;0)|)|Dw|^2\,dx\le 4\int_\Omega \frac{\mu_0^2}{\lambda_0}(|u(;0)|)|Du(;0)|^2w^2\,dx
}
Since $Du=D(u-s)$, the estimate \eqref{C2ak} implies the gradient is small:
\eqal{
\|Du(;0)\|_{L^\infty}\le C(s)\|f-s\|_{C^{2,\al}}\le C(s)\Pi_1.
}
If $\Pi_1\le 1$, then $\|f\|_{C^{2,\al}}\le |s|+C(s)$, and $C^{2,\al}$ estimate \eqref{C2a} yields
\eqal{
\|u(;0)\|_{L^\infty}\le C(s),\quad \lambda_0(|u(;0)|)\ge \lambda_0(C(s)),\quad \mu_0(|u(;0)|)\le \mu_0(C(s)),
}
for another $C(s)$.  Substituting these and the Poincar\'e inequality for $w\in W^{1,2}_0(\Omega)$ into \eqref{ineq4} yields
\eqal{
\int_\Omega|Dw|^2\,dx\le \frac{\mu_0^2}{\lambda_0^2}(C(s))\cdot C(s,\Omega)\cdot\Pi_1^2\int_\Omega|Dw|^2\,dx.
}
Choosing $\Pi_1$ small depending on $s,\Omega$ confirms that $w=0$.  The Fredholm alternative \cite[Theorem 5.3]{GT} shows that Dirichlet problem \eqref{probt} has a unique solution $v(;t)\in W^{1,2}(\Omega)$, and \cite[Corollary 8.7]{GT} yields the $W^{1,2}$ estimate
\eqal{
\|v(;t)\|_{W^{1,2}(\Omega)}\le C(\eqref{probt},\Omega)\|h\|_{W^{1,2}(\Omega)},
}
where $C(\eqref{probt},\Omega)$ indicates (uniform) dependence on the coefficients of the differential operator in \eqref{probt}.  For small $t$, the coefficients and their derivatives have controlled $t$ dependence in the $C^\beta$ norm, so we conclude that $v(;t)$ has a uniform-in-$t$ $W^{1,2}$ bound.  Moreover, since $f,h\in C^\infty$, standard Schauder bootstrapping yields higher regularity $C^{k,\al}(\pd\Omega)\to C^{k,\al}(\Omega)$ estimates for $u(;t)$ as in \eqref{C2a}, so the $t$ dependence of the coefficients in \eqref{probt} is uniform in any $C^k$ norm.   Therefore, the higher regularity $W^{k+2,2}(\Omega)$ estimate \cite[Theorem 8.13]{GT} combined with Sobolev embeddings yield uniform-in-$t$ estimates for $v(;t)$ in any $C^\ell$ norm.

\medskip
\textbf{Step 2:}

We now show $v(;t)$ converges to $v$ in $C^1$, where $v$ is the unique solution of Dirichlet problem \eqref{eqL}.  Note that the unique solvability for $v$ is similar to that for $v(;t)$ in Step 1, so we omit the proof.  In \eqref{probt}, as before, we split 
\eqal{
a_s(;t)=a_s(;0)+r_2(t),\quad a_{p_j}(;t)=a_{p_j}(;0)+r_3(t),\quad a(;t)=a(;0)+r_4(t),
}
where each $r_i(t)$ vanishes in $C^{1,\beta}(\Omega)$ as $t\to 0$.  Thus, omitting $t=0$ arguments, $v(;t)$ solves the inhomogeneous equation
\eqal{
\pd_i(a_sv(;t)u_i+a_{p_j}v_j(;t)u_i+av_i(;t))=o(1)_{C^{\beta}(\Omega)},
}
where $o(1)_{C^\beta(\Omega)}$ indicates a quantity vanishing as $t\to 0$ in $C^\beta(\Omega)$, and its uniformity in $t$ with respect to terms including $v$ or $v_i$ follows from Step 1.  We then put $w(;t)=v(;t)-v$ and see that $w(;t)$ uniquely solves the Dirichlet problem
\eqal{
\pd_i(a_sw(;t)u_i+a_{p_j}w(;t)_ju_i+aw(;t)_i)=o(1)_{C^{\beta}(\Omega)},\\
w(;t)|_{\pd\Omega}=0.
}
The right hand side clearly vanishes in $L^p$ for any $p>1$, so it follows from the strong $W^{2,p}$ estimate \cite[Theorem 9.13]{GT} that $\|w(;t)\|_{W^{2,p}(\Omega)}\to 0$ as $t\to 0$.  By Sobolev embedding, we deduce that $v(;t)\to v$ in $C^1$, as desired.
\end{proof}

We next linearize around solutions $u$ with large gradients if we assume that $a(s,p)$ satisfies additional conditions.

\begin{prop}
Let $a$ satisfy uniform lower bound \eqref{uni} and  decay condition \eqref{decay}.  Then for any solution $u$ to \eqref{eq} with smooth boundary data $f$, the linearization of \eqref{eq} exists at $u$.  
\end{prop}

\begin{proof}
The proof of Step 2 in Proposition \ref{prop:lin_small} is unchanged, and we repeat Step 1 until \eqref{aijbound}.  To proceed further, we invoke additional conditions \eqref{uni} and \eqref{decay}.  We suppose $t$ is sufficiently small so that
\eqal{
|a_s(;t)Du(;0)|\le\ep_0+o(1)_{C^{1,\beta}}\|Du(;0)\|_{L^\infty}\le 2\ep_0.
}
Applying the Cauchy-Schwartz inequality and the Poincar\'e inequality, we obtain
\eqal{
\int_\Omega|Dw|^2\,dx\le \frac{4\ep_0^2}{\lambda_0}\cdot C(\Omega)\int_\Omega |Dw|^2\,dx.
}
Choosing $\ep_0$ small depending on $\lambda_0,\Omega$ yields the desired conclusion.  The rest of the proof proceeds identically.
\end{proof}

\subsection{Solutions with prescribed boundary jets}
\label{sec:jet}



\medskip
We first introduce a coordinate normalization.
\begin{defn}\label{Omega}
We say that \textit{$\Omega$ sits above the origin}, if $\{x^n=0\}$ is a supporting hyperplane for $\Omega$ with $0\in\pd\Omega$,
and $\Omega\subset\{x^n>0\}$.
\end{defn}

The class of PDEs \eqref{eq} with conditions \eqref{aij} and \eqref{grow} considered is invariant under Euclidean isometries.  Let $I=R\circ T$ be a Euclidean isometry (rotation $R$ plus translation $T$) such that $I^{-1}(\Omega)$ sits above the origin.  The resulting PDE for $I^*u(x):=u(Ix)$ is also of class \eqref{eq}:
\eqal{\label{eqR}
\nabla\cdot \left[R_*a(I^*u,\nabla I^*u)\nabla I^*u\right]=0,\qquad x\in I^{-1}(\Omega),
}
where $R_*a(s,p)=a(s,R^{-1}p)$.  It is clear that $R_*a$ satisfies any of conditions \eqref{aij}, \eqref{grow}, or \eqref{decay} if and only if $a$ does.

\medskip
We now construct solutions with prescribed boundary jet, provided the gradient at the boundary point is sufficiently small compared to the solution.  No additional assumptions on $a$ are needed.  Recall $\Pi_1(s)$ as in Proposition \ref{prop:lin_small}.

\begin{prop}
\label{prop:open}
Let $\Pi:\re\to(0,\infty)$ be the positive continuous function defined in \eqref{Pidef}.  Then for any $x_0\in\pd\Omega$ and $(s,p)\in\re\times\re^n$ such that $|p|<\Pi(s)$, there exists boundary data $f\in C^\infty(\pd\Omega)$ for \eqref{eq} such that $u(x_0)=s$ and $\nabla u(x_0)=p$, with $\|f-s\|_{C^{2,\al}}\le \Pi_1(s)$.
\end{prop}

\begin{proof}
%

\textbf{Step 1.} We construct sub/super-solutions.

\medskip
First suppose $\Omega$ sits above the origin and $x_0=0$.  Given $(s,p)\in\re\times\re^n$ with $p=p'+p_ne_n$, we define
\eqal{
u^{s,p}(x):=s+p'\cdot x-Ap_n\log(1-x^n/A),
}
where $A$ is the sufficiently large constant
\eqal{\label{Adef}
A:=2\text{ diam }\Omega\ge 2\max_{x\in\Omega}x^n.
}
We have $u^{s,p}(0)=s$ and $Du^{s,p}(0)=p$.  The derivatives:
\eqal{
Du^{s,p}(x)=p'+\frac{p_n}{A-x^n}e_n,\quad D^2u^{s,p}(x)=\frac{p_n}{(A-x^n)^2}e_n\otimes e_n.
}
The expanded form of equation \eqref{eq} is
\eqal{\label{exp_eq}
a_{ij}u_{ij}+a_s|\nabla u|^2=0.
}
For $\pm 1=\text{sign }p_n$, we thus obtain
\eqal{
\pm (a_{ij}u^{s,p}_{ij}+a_s|Du^{s,p}|^2)&\ge \lambda_0(|u^{s,p}|)\frac{|p_n|}{(A-x^n)^2}-\mu_0(|u^{s,p}|)(|p'|^2+\frac{p_n^2}{(A-x^n)^2})\\
&\ge \frac{\mu_0(U')}{(A-x^n)^2}\left[\frac{\lambda_0(U')}{\mu_0(U')}|p_n|-(\frac{A^2}{4}|p'|^2+p_n^2)\right],
}
where 
$$
U':=s+\frac{A}{2}\,|p'|+A|p_n|\log 2
$$ 
satisfies $|u^{s,p}(x)|\le U'$ on $\Omega$.

\medskip
Let us now restrict $p$ to satisfy 
\eqal{\label{Udef}
\max(|p'|,|p_n|)\le 1,\qquad U'\le U:=s+2A.
}
Then we obtain
\eqal{
\pm (a_{ij}u^{s,p}_{ij}+a_s|Du^{s,p}|^2)&\ge\frac{\mu_0(U)}{(A-x^n)^2}\left[\frac{\lambda_0(U)}{\mu_0(U)}|p_n|-(\frac{A^2}{4}|p'|^2+p_n^2)\right].
}
Choosing now $p$ to be any vector defined in the paraboloid
\eqal{
\label{para0}
\left(\frac{2\lambda_0(U)}{A^2\mu_0(U)}\right)^{-1}|p'|^2\le|p_n|\le \frac{\lambda_0(U)}{2\mu_0(U)},\qquad \max(|p'|,|p_n|)\le 1,
}
we find that $u^{s,p}$ is a one-sided solution:
\eqal{
\pm (a_{ij}u^{s,p}_{ij}+a_s|Du^{s,p}|^2)&\ge 0.
}
To remove the ``corners" of the paraboloid \eqref{para0}, we denote by $U=U^s$ the constant in \eqref{Udef}, and denote the various constants by
$$
C^1_s:=\min(\frac{\lambda_0(U^{s})}{2\mu_0(U^{s})},1),\qquad C^2_s:=\min(\frac{2\lambda_0(U^{s})}{A^2\mu_0(U^{s})},1),
$$
so that after shrinking the bounds, the paraboloid becomes
\eqal{
\label{para1}
\frac{1}{C^2_s}|p'|^2\le |p_n|\le C^1_s.
}
To satisfy the smallness condition of Proposition \ref{prop:lin_small}, we shrink the bounds again.  Defining
$$
B^1_s:=\min(C^1_s,\frac{\Pi_1(s)}{N\text{ diam }\Omega}),\qquad B^2_s:=C^2_s,
$$
for some sufficiently large $N$, it is easy to verify that $u^{s,p}$ satisfies the desired estimate
$$
\|u^{s,p}-s\|_{C^{2,\al}(\pd\Omega)}\le \Pi_1(s)
$$
provided $p=(p',p_n)$ lies within the paraboloid
\eqal{
\label{para}
\frac{1}{B^2_s}|p'|^2\le |p_n|\le B^1_s.
} 
We next construct solutions with boundary jet inside this paraboloid, and more generally in the smallest rectangle containing it.

\medskip
\textbf{Step 2.} A comparison principle argument. 

\medskip
Given $s$, we let $p\in \re^{n-1}$ be any vector in the vertical boundary of the paraboloid \eqref{para}:
\eqal{\label{parabdy}
|p'|^2<B^1_sB^2_s,\qquad p_n=B^1_s.
}
Since $p_n>0$, it follows that $u^{s,p}$ is a subsolution of \eqref{eq}.  Choosing boundary data $f\in C^\infty(\pd\Omega)$ of the form
\eqal{
f^{s,p}(x):=u^{s,p}(x),\qquad x\in\pd\Omega,
}
the comparison principle \cite[Theorem 10.7]{GT} shows that
\eqal{
u[f^{s,p}](x)\ge u^{s,p}(x),\qquad x\in\Omega,
}
with equality at $x=0$, so the inner normal derivatives satisfy the same inequality:
\eqal{
\frac{\pd u[f^{s,p}]}{\pd x^n}(0)\ge \frac{\pd u^{s,p}}{\pd x^n}(0)=p_n.
}
If we instead choose 
$$
p_n=-B^1_s,
$$ 
then the inequalities reverse because $u^{s,p}$ becomes a supersolution.  This means that the set of normal derivatives with tangential jet $(s,p')$ and \textit{small} boundary data,
\eqal{
\mc N^{s,p'}:=\{\frac{\pd u[g]}{\pd x^n}(0)\Big| g\in C^\infty(\pd\Omega),g(0)=s,\nabla g(0)=p', \text{ and }\\
\|g-s\|_{C^{2,\al}(\pd\Omega)}\le \Pi_1(s)\},
}
contains numbers at least as large as $B^1_s$ and at least as small as $-B^1_s$.  Since the set of $g$ with boundary jet $(g,\nabla g)(0)=(s,p')$ is convex, it is connected.  So, of course, is the set of such $g$ also satisfying the smallness estimate
\eqal{
\label{gsmall}
\|g-s\|_{C^{2,\al}(\pd\Omega)}\le \Pi_1(s).
}
By the smoothness of solutions to \eqref{eq}, we deduce that $\mc N^{s,p'}$ is connected, and it therefore contains the interval $[-B^1_s,B^1_s]$.  Varying $p'$ according to \eqref{parabdy}, we conclude that for each $s$ and any vector $p$ in the rectangle
$$
|p'|^2\le B^1_sB^2_s,\qquad |p_n|\le B^1_s,
$$
there exists a solution with boundary jet $(u(0),\nabla u(0))=(s,p)$ with boundary data $g$ satisfying \eqref{gsmall}.  Defining $\Pi$ by the radius of the largest sphere inside the rectangle,
\eqal{\label{Pidef}
\Pi(s)=\min(\sqrt{B^1_sB^2_s},B^1_s).
}
This completes the proof of the Proposition, in the case that $\Omega$ sits above the origin.

\medskip
\textbf{Step 3.} Rotational invariance.

\medskip
Now suppose $\Omega\subset\re^n$ is any smooth domain, and choose $s,p$ such that $|p|<\Pi(s)$, and $x_0\in\pd\Omega$.  Letting $T(x)=x-x_0$ translate $x_0$ to the origin, we choose a rotation $R\in O(n)$ such that the inner normal vector is mapped to $e_n$, i.e. $R(-\nu|_{x_0})=e_n$.  Then for Euclidean isometry $I=R\circ T$, we see that $I^{-1}(\Omega)$ sits above the origin.  Moreover, inequalities \eqref{aij}, \eqref{coer}, and \eqref{grow} and the constant $A$ in \eqref{Adef} are preserved under the isometry.  Since $|Rp|<\Pi(s)$, we can find a solution $I^*u$ of the rotated equation \eqref{eqR} such that $I^*u(0)=s$ and $\nabla I^*u(0)=Rp$.  In original variables, we have $u(x_0)=s$ and $\nabla u(x_0)=p$, which completes the proof.
\end{proof}


If we also assume the decay condition \eqref{decay}, then we can choose $\Pi=+\infty$.  That is, we can find a solution with any prescribed boundary jet.  We can also weaken conditions \eqref{uni} and \eqref{decay}, in this proof, to
\eqal{
\label{decay1}
|a_s(s,p)|\le\,C\,\frac{\lambda(s,p)}{|p|},
}
where $\lambda(s,p)$ is in \eqref{aij}, and $C>0$ is arbitrary.  This condition makes \eqref{eq} behave similar to a linear equation which admits exponential solutions.
\begin{prop}
\label{prop:decay}
Let $a$ satisfy \eqref{decay1}, or in particular \eqref{uni} and \eqref{decay}.  For any $(s,p)\in\re\times\re^n$ and $x_0\in\pd\Omega$, there exists boundary data $f\in C^\infty(\pd\Omega)$ for \eqref{eq} such that $u(x_0)=s$ and $\nabla u(x_0)=p$.
\end{prop}

\begin{proof}
As in Step 3 in the proof of Proposition \ref{prop:open}, by the invariance of the problem under Euclidean isometries, it suffices to assume that $\Omega$ sits above the origin.

\medskip
Fix $(s,p)\in\re\times \re^{n}$ with $p_n\neq 0$ and $p':=p-p_ne_n$, and for $h>0$ small depending on the constant $C$ in \eqref{decay1} and $p$, we introduce the function
\eqal{
u^{s,p}(x):=s+p'\cdot x+h\,p_n\,(e^{x^n/h}-1).
}
If $p_n>0$ (resp. $p_n<0$) then we claim that $u^{s,p}$ is a subsolution (supersolution) of \eqref{eq}, with jet
\eqal{
u^{s,p}(0)=s,\qquad \nabla u^{s,p} (0)=p.
}
We have
\eqal{
|\nabla u^{s,p}|^2=|p'|^2+p_n^2e^{2x^n/h},\qquad u^{s,p}_{ij}=\frac{p_n}{h}e^{x^n/h}\delta_{in}\delta_{jn},
}
so ellipticity \eqref{aij} yields, for $\pm 1=\text{sign } p_n$,
\eqal{
\pm(a_s|\nabla u^{s,p}|^2+a_{ij}u^{s,p}_{ij})\ge \pm a_s|\nabla u^{s,p}|^2+\lambda \frac{|p_n|}{h}e^{x^n/h}.
}
Applying decay condition \eqref{decay1}, we obtain
\eqal{
\pm(a_s|\nabla u^{s,p}|^2+a_{ij}u^{s,p}_{ij})\ge \lambda e^{x^n/h}\left(\frac{|p_n|}{h}-C\sqrt{|p'|^2e^{-2x^n/h}+p_n^2}\right).
}
Choosing $h$ such that $|p_n|/h\ge C\sqrt{|p'|^2+p_n^2}$ verifies the claim.

\medskip
We next observe that by continuity, the set of possible normal derivatives with tangential jet $(s,p')$
\eqal{
\{\frac{\pd u[g]}{\pd x^n}(0)|\quad g\in C^\infty(\pd\Omega) \text{ and }g(0)=s,\nabla'g(0)=p'\}
}
is connected and nonempty, so to show it is $\re$, we need only show it is unbounded from above and below.  If $p_n\gg 1$ is large, we let $f=u^{s,p}|_{\pd\Omega}$ and $u$ solve \eqref{eq} with boundary value $f$.  Then $u^{s,p}$ is a subsolution, and the comparison principle \cite[Theorem 10.7]{GT} says 
$$
u(x)\ge u^{s,p}(x),\qquad x\in\Omega,
$$
with equality at $x=0$, so 
$$
\frac{\pd u}{\pd x^n}(0)\ge \frac{\pd u^{s,p}}{\pd x^n}(0)=p_n.
$$
Repeating the argument with $p_n\ll -1$ and using supersolution $u^{s,p}$ yield the desired conclusion.
\end{proof}

\begin{rem}
The results of this section hold true unchanged for non-homogeneous conductivities $a=a(x,s,p)$, provided the inequalities \eqref{aij} and \eqref{grow} remain independent of $x$.  We instead define the transformed conductivity in \eqref{eqR} by $I_*a(x,s,p)=a(Ix,s,R^{-1}p)$.  
\end{rem}

\subsection{Geometric reformulation of the linearized problem}
\label{sec:geo}
We next geometrically reformulate the linearized equation \eqref{eqL} and its DN map \eqref{DNL}, see e.g. \cite{LU} for the same with the conductivity equation via the Laplace operator.  We assume throughout this section that the linearization of \eqref{eq} exists at a given solution $u$, which we addressed in Section \ref{sec:lin}.

\medskip
Let us first rewrite the equation.  We can decompose the principal part of \eqref{eqL}, namely
$$
\pd_i(a_{p_j}v_ju_i+av_i),
$$
into its symmetric and anti-symmetric parts and rewrite the linearized differential operator as:
\eqal{\label{eqL1}
\mc L_a[u]v&:=\pd_i(a_svu_i+a_{p_j}v_ju_i+av_i)\\
&=\pd_i(a_{ij}v_j+b^iv),
}
where linearized conductivity $a_{ij}$, given in \eqref{aij}, is the symmetric part, and 
\eqal{\label{bA}
b^i:=a_su_i-\pd_jA_{ij},\qquad A_{ij}:=\frac{1}{2}(a_{p_j}u_i-a_{p_i}u_j)
}
accounts for the anti-symmetric part.  We can also rewrite the linearized DN map \eqref{DNL}:
\eqal{\label{DNL1}
\Gamma_a[u](h)=(a_{ij}\nu_iv_j+A_{ij}\nu_iv_j+a_sv\,u_in_i)dS.
}
Henceforth in this section only, we distinguish Euclidean coordinates from ordinary subscripts (which will denote local coordinates) unless otherwise stated.

\medskip
We next rewrite this equation in terms of geometric quantities.  Let $x^i$ denote local coordinates on $\Omega$, and denote $x^i_e$ to be global Euclidean coordinates.  We define a $(2,0)$ tensor field $G\in \mc T^{2}_0(\Omega)$ by its values in Euclidean coordinates, using the symmetric part $a_{ij}$ in \eqref{aij} of the linearized operator.  If $n\ge 3$,
\eqal{
G^{ij}_e:=G(dx^i_e,dx^j_e)= 
(\det a_{k\ell})^{\frac{1}{2-n}}a_{ij},\qquad G=G^{ij}\pd_i\pd_j.
}
If $n=2$, we first define a conductivity function $\sigma\in C^\infty(\Omega)$ by
\eqal{\label{sigma}
\sigma:=\det a_{k\ell}.
}
We can then define $G$ as follows:
\eqal{
G^{ij}_e:=G(dx^i_e,dx^j_e)= (\det a_{k\ell})^{-1/2}a_{ij},\qquad G=G^{ij}\pd_i\pd_j.
}
In each case, the relationship can be written as
\eqal{\label{gaij}
&n\ge 3:&&\frac{1}{\sqrt{\det G^{k\ell}_e}}G^{ij}_e=a_{ij},\\
&n=2:&&\sigma\,G^{ij}_e=a_{ij},\qquad \det G^{k\ell}_e=1.
}
We next define a Riemannian metric $g\in \mc T^0_2(\Omega)$ by the inverse of $G$ in Euclidean coordinates, which is well defined by ellipticity \eqref{aij}
\eqal{\label{g}
g_{e,ij}:=g(\frac{\pd}{\pd x^i_e},\frac{\pd}{\pd x^j_e})=(G_e^{-1})_{ij},\qquad g=g_{ij}dx^idx^j.
}
We will instead denote $\pd_{e,i}=\pd/\pd x^i_e$.  The divergence of the linearized conductivity takes two forms in various dimensions:
\eqal{\label{divg}
&n\ge 3:&& \pd_{e,i}(a_{ij}\pd_{e,j}v)=\sqrt{g_e}\,\Delta_{g}v,\\
&n=2:&&\pd_{e,i}(a_{ij}\pd_{e,j}v)=\text{div}_g(\sigma\nabla_gv)=\sigma\Delta_gv+\langle \nabla_g\sigma,\nabla_gv\rangle_g,
}
where, in local coordinates, the Laplace operator, gradient, and divergence are given by
\eqal{
\Delta_gv=\frac{1}{\sqrt g}\pd_i(\sqrt g\,g^{ij}\pd_jv),\quad \nabla_gv=g^{ij}\pd_jv\,\pd_i,\quad\text{div}_gA=\frac{1}{\sqrt g}\pd_i(\sqrt g\,A^i).
}
To account for the lower order terms in \eqref{eqL1}, we define the \textit{magnetic Schr\"odinger operator with potential}:
\eqal{\label{mag}
\Delta_{g,A,q}v&:=\frac{1}{\sqrt g}(\pd_i+A_i)\sqrt g\,g^{ij}(\pd_j+A_j)v+qv\\
&=\Delta_gv+2\langle A^\#,\nabla_gv\rangle+\left( 
\text{div}_gA^\#+|A|_g^2+q\right)v,
}
where $A=A_idx^i,q$ are to be chosen, $A^\#=g^{ij}A_jdx^i$, and $|A|_g^2=g^{ij}A_iA_j$.  Recalling \eqref{divg}, we obtain the following relations:
\eqal{
\label{eqLgeo}
&n\ge 3:&&\mc L_a[u]v=\sqrt{g_e}\,\Delta_{g,A,q}v,\\
&n=2:&&\mc L_a[u]v=\sigma\,\Delta_{g,A,q}v,
}
provided the following relations are true:
\eqal{\label{Aq}
&n\ge 3:&&b^i=2\sqrt{g_e}\,g^{ij}_eA_{e,j},\qquad \pd_{e,i}b^i=\sqrt{g_e}(\text{div}_gA^\#+|A|_g^2+q),\\
&n=2:&&b^i+g^{ij}_e\pd_{e,j}\sigma=2\sigma g^{ij}_eA_{e,j},\qquad \pd_{e,i}b^i=\sigma(\text{div}_gA^\#+|A|_g^2+q).
}
This system can be solved for $A$ and $q$ in terms of $b,\sigma,$ and $g$.  This completes the reformulation of the linearized PDE \eqref{eqL} via \eqref{eqLgeo}.

\medskip
We next reformulate the linearized DN map \eqref{DNL1}, and start with the anti-symmetric part of the linearized operator \eqref{bA}.  We define a $(1,1)$ tensor $\al\in \mc T^1_1(\Omega)$ in terms of $g$:
\eqal{\label{alA}
\al_{e,j}^i:=\al(dx^i_e,\pd_{e,j})= \frac{1}{\sqrt{g_e}}\,g_{e,jk}A_{ik}
,\qquad \al=\al^i_j\pd_idx^j.
}
Here, $\sqrt{g_e}=\sqrt{\det g_{e,ij}}$, and $A_{ik}$ is as in \eqref{bA}.  The tensor $\al$ is anti-symmetric in the inner product induced by $g$:
\eqal{\label{anti}
\langle \al\cdot V,W\rangle_g=-\langle V,\al\cdot W\rangle_g,\quad V,W\in T\Omega,
}
since in Euclidean coordinates, with $V=V^i_e\pd_{e,i},W=W^i_e\pd_{e,i}$, we have
\eqal{
\langle \al\cdot V,W\rangle_g&=g_{e,i\ell}\left(\frac{1}{2\sqrt{g_e}}g_{e,jk}(a_{p_k}u_i-a_{p_i}u_k)V^j_e\right)W^\ell_e\\
&=:(g_{e,i\ell}T^{ik}g_{e,jk})V^j_eW^\ell_e,
}
where $T^{ik}$ is anti-symmetric in its indices, so the claim follows.  

\medskip
The Riemannian volume form $dV_g$ has the Euclidean representation
\eqal{
dV_g=\sqrt{\det g_e}\,dx^1_e\wedge\cdots\wedge dx^n_e=\sqrt{\det g_e}\,dx.
}
The induced area form $dS_g$ on $\pd \Omega$ is
\eqal{
dS_g=\nu_g\,\lrcorner\, dV_g|_{\pd\Omega},
}
where $\nu_g$ is the unit normal to $\pd\Omega$.  Letting $f$ be a defining function for $\pd\Omega$, the unit normals for Euclidean and $g$ metrics can be written in Euclidean coordinates as
\eqal{
\nu=(|df|_e)^{-1}\pd_{e,i}f\pd_{e,i},\qquad \nu_g=(G(df,df))^{-1/2}G^{ij}_e\pd_{e,j}f\pd_{e,i}.
}
We have the following relationship between the normal vectors.  This allows us to reformulate the part of the linearized DN map with anti-symmetric part $A_{ij}$.  This is likely well known, but we give the proof for completeness.  

\begin{lem}\label{lem:normal}
As sections of the product bundle 
$ 
\pd(T\Omega)\otimes \Lambda^0_{n-1}(\pd\Omega)$, the following equality holds:
\eqal{\label{product}
&n\ge 3:&&(a_{ij}\nu_j\pd_i)\otimes dS=\nu_g\otimes dS_g,\\
&n=2:&& (a_{ij}\nu_j\pd_i)\otimes dS=\sigma\,\nu_g\otimes dS_g.
}
In particular,
\eqal{\label{product_cases}
&n\ge 3:&&(a_{ij}\nu_iv_j+A_{ij}\nu_iv_j)dS=\langle \nabla_g v+\al\cdot\nabla_gv,\nu_g\rangle_g dS_g,\\
&n=2:&& (a_{ij}\nu_iv_j+A_{ij}\nu_iv_j)dS=\langle \sigma\,\nabla_g v+\al\cdot\nabla_gv,\nu_g\rangle_g dS_g,
}
where $v\in C^\infty(\ov\Omega)$, $a_{ij},\sigma$ are in \eqref{gaij}, $A_{ij}$ is in \eqref{bA}, and $\al$ is in \eqref{alA}. 
The local coordinates used in this lemma and proof are Euclidean, without subscript $e$.
\end{lem}

\begin{proof}
We will show the identities at each $x_0\in\pd\Omega$.  The setting of the lemma is unchanged after applying a Euclidean isometry because the class of PDEs \eqref{eq} is unchanged by such.  After a Euclidean isometry, we may assume that $x_0=0$ is the origin, and that $\Omega$ sits above the origin, recall Definition \ref{Omega}.  The tangent space $T_0(\pd\Omega)$ at the origin is the span of $\pd_1,\dots,\pd_{n-1}$, and $f=x^n$ is a defining function for $\pd\Omega$ at $x=0$, so the unit normal vectors are
\eqal{
\nu|_{x=0} = \pd_n,\qquad \nu_g|_{x=0}=(G_{nn})^{-1/2}G^{in}\pd_i.
}
Since $\pd_i\,\lrcorner \,dV_g|_{x^n=0}=0$ for $i<n$, this means the area forms at $x=0$ are
\eqal{
dS|_{x=0}=dx^1\wedge\cdots\wedge dx^{n-1},\qquad dS_g|_{x=0}=\sqrt{G_{nn}}\sqrt{g}\,dx^1\wedge\cdots \wedge dx^{n-1}.
}
We find that
\eqal{
\nu_g\otimes dS_g|_{x=0}=(\sqrt{g}\,G^{in}\pd_i)\otimes dS,
}
while \eqref{gaij} yields
\eqal{
&n=3:&& (a_{ij}\nu_j\pd_i)\otimes dS|_{x=0}=(\sqrt{g}\,G^{in}\pd_i)\otimes dS,\\
&n=2:&& (a_{ij}\nu_j\pd_i)\otimes dS|_{x=0}=(\sigma\, G^{in}\pd_i)\otimes dS,
}
with $\det {g}=1$ if $n=2$, which verifies \eqref{product}.  

\medskip
That the $a_{ij},\nabla_gv$ terms in \eqref{product_cases} are equal follows from left-contracting \eqref{product} against the $(0,1)$ tensor $dv$.  If we instead contract against $(\al\cdot\nabla_gv)_\flat$, where $A_\flat=g_{ij}A^jdx^i$, and if we recall \eqref{alA}, then we obtain for $n=2$,
\eqal{\label{Aij_calc}
\langle \al\cdot\nabla_gv,\nu_g\rangle dS_g&=\frac{1}{\sigma}\,g_{ik}(\al^k_\ell g^{\ell m}v_m)a_{ij}\nu_jdS\\
&=g_{ik}(g_{\ell s}A_{ks}g^{\ell m}v_m)g^{ij}\nu_jdS\\
&=A_{js}v_s\nu_j dS.
}
The $n\ge 3$ case is similar.  This completes the proof.
\end{proof}

For the final term in \eqref{DNL1}, we define a vector field $\beta=\beta^i\pd_i$ so that the following equation holds on $\pd\Omega$:
\eqal{\label{beta}
&n\ge 3:&& (a_su_in_i)dS=\langle A^\#+\beta,\nu_g\rangle dS_g,\\
&n=2:&&(a_su_in_i)dS=\langle \sigma A^\#+\beta,\nu_g\rangle dS_g.
}
We now recall the \textit{magnetic Dirichlet-to-Neumann map}, $C^\infty(\pd\Omega)\to C^\infty(\pd\Omega):$
\eqal{\label{Lambda}
&\Lambda_{g,A,q}(h):&&h\mapsto \langle \nabla_gv[h]+A^\#v[h],\nu_g\rangle,\\
&\text{where:}&&\Delta_{g,A,q}v=0,\\
& &&v|_{\pd\Omega}=h.
}
Then we have the following geometric reformulation.  The linearized PDE \eqref{eqL1} corresponds to a magnetic Schr\"odinger operator with potential, and the linearized DN map \eqref{DNL1} corresponds to an oblique magnetic normal derivative with lower order term.

\begin{prop}\label{prop:geo}
Let $u$ be a solution of \eqref{eq} and $f\in C^\infty(\pd\Omega)$ such that the linearization of \eqref{eq} exists at $u$.  Let also $h\in C^\infty(\pd\Omega)$ and $v$ solve the linearized problem \eqref{eqL}.

\medskip
\noindent
Then for $n\ge 3:$
\eqal{\label{n3}
\mc L_a[u]v&=\sqrt{g_e}\,\Delta_{g,A,q}v,\\
\Gamma_a[u](h)&=\Lambda_{g,A,q}(h)dS_g+\langle\al\cdot\nabla_gh+\beta h,\nu_g\rangle_g dS_g.
}
For $n=2$:
\eqal{\label{n2}
\mc L_a[u]v&=\sigma\,\Delta_{g,A,q}v,\\
\Gamma_a[u](h)&=\sigma\Lambda_{g,A,q}(h)dS_g+\langle\al\cdot\nabla_gh+\beta h,\nu_g\rangle_g dS_g.
}
Here, metric $g$ is given by \eqref{gaij},\eqref{g}, and $\sigma$ by \eqref{sigma} for $n=2$.  Magnetic potential $A$ and potential $q$ are determined by \eqref{Aq}, anti-symmetric tensor $\al$ is given by \eqref{alA}, and vector field $\beta$ is defined according to \eqref{beta}.  The map $\Lambda$ is defined by \eqref{Lambda}, and $\Gamma_a[u]$ is defined by \eqref{DNL1}.

\medskip
\noindent
In particular, if $\Gamma_a=\Gamma_{\ti a}$, then for $n\ge 3$:
\eqal{\label{DNn3}
\Lambda_{g,A,q}(h)dS_g+\langle\al\cdot\nabla_gh+\beta h,\nu_g\rangle_g dS_g=\\
\Lambda_{\ti g,\ti A,\ti q}(h)dS_{\ti g}+\langle\ti \al\cdot\nabla_{\ti g}h+\ti \beta h,\nu_{\ti g}\rangle_{\ti g} dS_{\ti g}.
}
For $n=2:$
\eqal{\label{DNn2}
\sigma\Lambda_{g,A,q}(h)dS_g+\langle\al\cdot\nabla_gh+\beta h,\nu_g\rangle_g dS_g=\\
\ti\sigma \Lambda_{\ti g,\ti A,\ti q}(h)dS_{\ti g}+\langle\ti \al\cdot\nabla_{\ti g}h+\ti \beta h,\nu_{\ti g}\rangle_{\ti g} dS_{\ti g}.
}
\end{prop}

\begin{proof}
The only claim remaining to prove is the identity
\eqal{
\langle \al\cdot\nabla_gv[h],\nu_g\rangle_g dS_g=\langle \al\cdot\nabla_gh,\nu_g\rangle_g dS_g,
}
or that this term depends only on the tangential projection of $\nabla_gv[h]$.  This follows from the anti-symmetry of $\al$, since the normal contribution is
$$
\langle\nabla_gv[h],\nu_g\rangle\langle\al\cdot\nu_g,\nu_g\rangle=0,
$$
which follows from \eqref{anti}.
\end{proof}

\subsection{Boundary determination of the linearized conductivity}
\label{sec:bdy}
In this section, we determine the tangential part of the linearized conductivity $a_{ij}$ at the boundary $\pd\Omega$.  The main technical tool we need is the calculation of the symbol of the magnetic DN map \eqref{Lambda} in boundary normal coordinates, \cite[Lemma 8.6]{DKSU}.  See \cite{LU} for the construction in the case of the ordinary Laplace operator.

\medskip
Letting $x^1,\dots,x^{n-1}$ be local coordinates on a neighborhood $\Sigma\subset\pd\Omega$, we denote boundary normal coordinates in a neighborhood of $\Sigma$ in $\Omega$ by $(x^1,\dots,x^n)$, with $x^n$ the $g$-distance function from $\pd\Omega$.  The metric and its dual in boundary normal coordinates are
\eqal{
g=(dx^n)^2+\sum_{i,j<n}g_{ij}dx^idx^j,\quad G=\pd_n\otimes \pd_n+\sum_{i,j<n}g^{ij}\pd_i\otimes \pd_j.
}
Note that 
\eqal{
\sum_{i,j<n}g_{ij}dx^idx^j=g|_{T\pd\Omega},\qquad \sum_{i,j<n}g^{ij}\pd_i\otimes \pd_j=G|_{T^*\pd\Omega},
}
if $x\in\pd\Omega$.  For two metrics $g,\tilde g$, we use the \textit{same} local coordinates $x^1,\dots,x^{n-1}$ for $\Sigma$ when defining boundary normal coordinates.  

\medskip
We also need the principal symbol of pseudo-differential operators on a manifold using left quantization, see e.g. \cite[Definition 18.1.20]{H}.  The differential operator $C^\infty(\pd\Omega)\to C^\infty(\pd\Omega)$ defined by
$$
h\mapsto\langle\al\cdot\nabla_gh+\beta h,\nu_g\rangle_g=\sum_{k<n}(-\al\cdot\nu_g)^kh_k+\langle\beta,\nu_g\rangle h,
$$
clearly has principal symbol
\eqal{
i\langle\al\cdot\xi^\#,\nu_g\rangle_g=\sum_{k<n}(-\al\cdot\nu_g)^k(i\xi)_k,\qquad (x,\xi)\in T^*_x(\pd\Omega).
}
To compute the final term in the DN maps \eqref{n3},\eqref{n2}, we now apply \cite[Lemma 8.6]{DKSU}, see also \cite{LU}, where the symbol of the magnetic DN map \eqref{Lambda} is computed in boundary normal coordinates in $n\ge 3$ dimensions.  The proof there goes through unchanged for the case of $n=2$, although the subsequent boundary determination problem is different due to conformal invariance; see the remark in \cite{LU}.

\begin{lem}[From Lemma 8.6 in \cite{DKSU}]
Let $(\Omega,g)$ be a compact Riemannian manifold with boundary with dimension $n\ge 2$.  Then $\Lambda_{g,A,q}$ is a pseudo-differential operator of order 1 on $\pd \Omega$, and its principal symbol is
\eqal{
-|\xi|_g:=-\sqrt{\sum_{i,j<n}g^{ij}(x)\xi_i\xi_j},\qquad (x,\xi)\in T^*_x\pd \Omega.
}
\end{lem}

\medskip
We thus obtain the principal symbol of $\Gamma_a[u]$, thought of as a form-valued pseudo-differential operator:
\eqal{\label{sym}
&n\ge 3:&&(-|\xi|_g+i\langle\al\cdot\xi^\#,\nu_g\rangle)dS_g,\\
&n=2:&&(-\sigma|\xi|_g+i\langle\al\cdot\xi^\#,\nu_g\rangle)dS_g.
}

In dimension $n\ge 3$, we can recover the tangential projection of the metric on $\pd\Omega$.  In dimension $n=2$, we can recover the ``effective conductivity" $\sigma$ on the boundary.  In both cases, we can recover the normal elements of the anti-symmetric tensor.

\begin{lem}\label{lem:O1}
If $\Gamma_a=\Gamma_{\ti a}$, then $G|_{T^*\pd\Omega}=\ti G|_{T^*\pd\Omega}$ if $n\ge 3$, where $G$ is given in terms of $a$ in \eqref{gaij}.  If $n=2$, then $\sigma=\ti\sigma$ on $\pd\Omega$, where we defined $\sigma$ by \eqref{sigma}.  In both cases, we have
\eqal{\label{alN}
(\al\cdot\nu_g)\,dS_g=(\ti\al\cdot\nu_{\ti g})\,dS_{\ti g}
} 
on $\pd\Omega$.
\end{lem}

\begin{proof}
By Proposition \eqref{prop:geo}, equalities \eqref{n3}, \eqref{n2} hold, so the principal symbols on $T^*\pd\Omega$ of the associated pseudo-differential operators on $\pd\Omega$ must coincide.  By \eqref{sym} and \eqref{DNn3}-\eqref{DNn2}, we obtain
\eqal{\label{sym1}
&n\ge 3:&&(-|\xi|_g+i\langle\al\cdot\xi^\#,\nu_g\rangle)dS_g=(-|\xi|_{\ti g}+i\langle\ti \al\cdot\ti\xi^\#,\nu_{\ti g}\rangle)dS_{\ti g},\\
&n=2:&&(-\sigma|\xi|_g+i\langle\al\cdot\xi^\#,\nu_g\rangle)dS_g=(-\ti \sigma|\xi|_{\ti g}+i\langle\ti \al\cdot\ti \xi^\#,\nu_{\ti g}\rangle)dS_{\ti g},
}
where we denote $\ti\xi^\#=\ti g^{ij}\xi_j\pd_i$.  In boundary normal coordinates, we have 
$$
dS_g=\pd_n\,\lrcorner\, dV_g|_{x^n=0}=\sqrt{g}\, dx^1\wedge\cdots\wedge dx^{n-1}=\sqrt{g|_{T\pd\Omega}}\,dx^1\wedge\cdots\wedge dx^{n-1},
$$
so equating the real parts in \eqref{sym1} yields
\eqal{\label{real}
&n\ge 3:&&(\det g) g^{ij}=(\det \ti g)\ti g^{ij},\qquad i,j<n,\\
&n=2:&&\sigma=\tilde\sigma.
}
The imaginary parts in \eqref{sym1} yield \eqref{alN} in both cases.  

\medskip
Taking the determinant of both sides of \eqref{real} shows that $(\det g)^{n-2}=(\det \ti g)^{n-2}$, hence $g^{ij}=\ti g^{ij}$ for $i,j<n$.  The $n\ge 3$ case thus recovers $G|_{T^*\pd\Omega}$.
\end{proof}

We thus obtain the main result of this section, a boundary determination result for the quasilinear PDE \eqref{eq} in Euclidean coordinates. 
\begin{prop}
\label{prop:bdy}
Let $u$ be a solution of \eqref{eq} and $f\in C^\infty(\pd\Omega)$ such that the linearization of \eqref{eq} exists at $u$.   Similarly with $\ti u$, and $\ti a$ replacing $a$.

\medskip
If $n\ge 3$ and $\Gamma_a=\Gamma_{\ti a}$, then the following holds on $\pd\Omega$:
\begin{align}
\label{symeq}
[aI+\frac{1}{2}(\nabla_pa\otimes \nabla f+&\nabla f\otimes\nabla_pa)]\Big|_{T^*\pd\Omega}=
\\
\nonumber
&[\ti aI+\frac{1}{2}(\nabla_p\ti a\otimes \nabla  f+\nabla  f\otimes\nabla_p\ti a)]\Big|_{T^*\pd\Omega}.
\end{align}


If $n\ge 2$ and $\Gamma_a=\Gamma_{\ti a}$, then we have on $\pd\Omega$,
\eqal{
\label{aneq}
(\nabla_pa\otimes\nabla u-\nabla u\otimes\nabla_pa)&\cdot \nu=(\nabla_p\ti a\otimes\nabla \ti u-\nabla \ti u\otimes\nabla_p\ti a)\cdot \nu.
}

\medskip
If $n=2$ and $\Gamma_a=\Gamma_{\ti a}$, then we also have 
\eqal{\label{deteq}
\det a_{ij}=\det \ti a_{ij}
} 
on $\pd\Omega$, where $a_{ij}$ is in \eqref{aij}.
\end{prop}

\begin{proof}
If $\Gamma_a=\Gamma_{\ti a}$, then Lemma \ref{lem:O1} applies.  Since for $n\ge 3$ we have $G|_{T^*\pd\Omega}=\ti G|_{T^*\pd\Omega}$, the characterization \eqref{gaij} of $a_{ij}$ implies \eqref{symeq}; indeed, the restriction of $\nabla u$ to $T^*\pd\Omega$ is precisely $\nabla f$ by \eqref{eq}.  Recalling the definition \eqref{sigma} of $\sigma$ shows that $\det a_{ij}=\det \ti a_{ij}$ on $\pd\Omega$.

\medskip
For \eqref{alN}, we recall Lemma \ref{lem:normal} and \eqref{Aij_calc}, in the form
\eqal{
-(\al\cdot\nu_g)^sdS_g=(A_{js}\nu_j)dS,
}
and \eqref{alN} yields
\eqal{
A_{ij}\nu_j=\ti A_{ij}\nu_j,
}
which is precisely \eqref{aneq}, recalling the definition \eqref{bA} of $A_{ij}$.
\end{proof}



\section{Proofs of Theorems}
\label{sec:proof}

Suppose first that $\Omega$ sits above the origin, see Definition \ref{Omega}.  Evaluating equalities \eqref{symeq}, \eqref{aneq}, \eqref{deteq} at $x=0$ and identifying $T^*_0\pd\Omega$ with $\re^{n-1}\times\{0\}$, we obtain the following equalities at $x=0$:

\medskip
\noindent
If $n\ge 3:$
\begin{align}
\label{geq1}
a\delta_{ij}+\frac{1}{2}(a_{p_j}f_i+a_{p_i}f_j)=\ti a\delta_{ij}+\frac{1}{2}(\ti a_{p_j}f_i+\ti a_{p_i}f_j),\quad i,j<n.
\end{align}
If $n\ge 2:$
\begin{align}
\label{aleq1}
a_{p_i}u_n-a_{p_n}f_i=\ti a_{p_i}\ti u_n-\ti a_{p_n}f_i,\quad i<n.
\end{align}
If $n=2:$
\eqal{
\label{deteq1}
(a+a_{p_1}f_1)(a+a_{p_2}u_2)-\frac{1}{4}(a_{p_1}u_2&+a_{p_2}f_1)^2=(\ti a+\ti a_{p_1}f_1)(\ti a+\ti a_{p_2}\ti u_2)\\
&-\frac{1}{4}(\ti a_{p_1}\ti u_2+\ti a_{p_2}f_1)^2.
}
Our approach is to use convenient choices of $f$ to deduce $a=\ti a$ along a subspace of $\re\times\re^n$, then use Euclidean isometry invariance to rotate this subspace and deduce the full result.  The homogeneity assumption $\pd a(x,s,p)/\pd x=0$ is used in the last step, but we remark in each case how to weaken it to obtain Corollary \ref{cor:x}.  Theorem \ref{thm:decay} is an immediate consequence of the below proof once we replace $\Pi(s)$ with $+\infty$.

\bigskip
\textbf{The case $n\ge 3$}

\medskip
Fixing $s\in\re$, we choose $p\in\re^{n-1}\times\{0\}$ such that $0<|p|<\Pi(s)$.  Using Proposition \ref{prop:open}, we can find $f\in C^\infty(\pd\Omega)$ such that the solution $u=u[f]$ to \eqref{eq} has the following properties:
\eqal{
f(0)=s,\quad \nabla f(0)=p,\quad \frac{\pd u}{\pd x^n}(0) =0,\quad \|f-s\|_{C^{2,\al}(\pd\Omega)}\le\Pi_1(s).
}
By Proposition \ref{prop:lin_small}, the linearization thus exists at $u=u[f]$, and we can invoke Proposition \ref{prop:bdy}.

\smallskip
Since $\Gamma_a=\Gamma_{\ti a}$, we infer at $x=0$,
\eqal{
\ti a \frac{\pd\ti u}{\pd x^n}(0)=a\frac{\pd u}{\pd x^n}(0)=0,
}
and hence $\ti u_n(0)=0$ as well.  Substituting this into \eqref{geq1} yields
\eqal{
a\delta_{ij}+\frac{1}{2}(a_{p_j}p_i+a_{p_i}p_j)|_{(s,p)}=\ti a\delta_{ij}+\frac{1}{2}(\ti a_{p_j}p_i+\ti a_{p_i}p_j)|_{(s,p)},\quad i,j<n.
}
Let us put $\mc A=a-\ti a$:
\eqal{
\label{geq3}
\mc A\delta_{ij}+\frac{1}{2}(\mc A_{p_j}p_i+\mc A_{p_i}p_j)|_{(s,p)}=0,\quad i,j<n.
}
The spectrum of this matrix, listed in increasing order, is
\eqal{\label{spec}
\mc A+\frac{1}{2}p\cdot\nabla'_p\mc A-\frac{1}{2}|p||\nabla_p'\mc A|,\quad \mc A,\dots,\mc A,\quad \mc A+\frac{1}{2}p\cdot\nabla'_p\mc A+\frac{1}{2}|p||\nabla_p'\mc A|,
}
where $\nabla_p'\mc A=(\mc A_{p_1},\dots,\mc A_{p_{n-1}})$, and middle eigenvalue $\mc A$ is listed $n-3$ times.  From \eqref{geq3}, these eigenvalues are all zero.  Subtracting the largest and smallest eigenvalues yields
\eqal{
|\nabla_p'\mc A(s,p)|=0,\qquad s\in\re,\quad p\in \re^{n-1}\times\{0\}, \quad |p|<\Pi(s)
}
so the smallest eigenvalue yields
\eqal{\label{mcA0}
\mc A(s,p)=0,\qquad s\in\re,\quad p\in \re^{n-1}\times\{0\}, \quad|p|<\Pi(s).
}
Now choose any Euclidean isometry $I=R\circ T$ such that $I^{-1}(\Omega)$ sits above the origin.  Then the above argument applies to the rotated PDE \eqref{eqR}, and we obtain
\eqal{
R_*\mc A(s,p)=0,\qquad s\in\re,\quad p\in \re^{n-1}\times\{0\}, \quad|p|<\Pi(s).
}
The set of such rotations $R$ is all of $O(n)$.  Recalling that 
$$
R_*\mc A(s,p)=\mc A(s,R^{-1}p)
$$ 
and varying $R\in O(n)$, we conclude that 
\eqal{
\label{last_step}
\mc A(s,p)=0,\qquad s\in\re,\quad p\in \re^{n}, \quad|p|<\Pi(s).
}
This completes the proof of Theorem \ref{thm:open} if $n\ge 3$.

\begin{rem}
\label{rem:n3}
Theorem \ref{thm:open} still holds if $n\ge 3$ and we replace $a(s,p)$ and $\ti a(s,p)$ by 
$$
F(x,s,p,b(s,p)\quad \text{ and }\quad F(x,s,p,\ti b(s,p)),
$$ 
where $t\mapsto F(x,s,p,t)$ is injective for each fixed $(x,s,p)$.  To see this, we merely observe that \eqref{mcA0} and the injectivity of $F$ imply
$$
b(s,p)=\ti b(s,p),\qquad s\in\re,\quad p\in\re^{n-1}\times\{0\},\quad |p|<\Pi(s).
$$
The Euclidean invariance argument completes the proof.
\end{rem}

\begin{rem}
\label{rem:bsym}
Another non-homogeneous extension.  Let $a$ be non-homogeneous with a symmetry in one direction:
$$
a=a(x^1,\dots,x^{n-1},s,p).
$$
If $\Omega$ is convex, then Theorem \ref{thm:open} still holds.  Indeed, instead choosing a general vector  $p\in\re^n$ with $|p|<\Pi(s)$, we simply repeat the proof starting at \eqref{geq3} and deduce 
$$
\mc A(0,0,\dots,0,s,p)=0
$$
directly.  Since $\Omega$ is convex, we can place $\Omega$ above a hyperplane at any point $x\in\pd\Omega$, so a Euclidean invariance argument means there holds
$$
\mc A(\pi_{\re^{n-1}\times\{0\}}(x),s,p)=0
$$
for any $x\in\pd\Omega$.  Since $\pd\mc A/\pd x^n=0$, this completes the proof.  It is not clear whether this works for $n=2$ or $\Omega$ non-convex.
\end{rem}

\bigskip
\noindent
\textbf{The case $n=2$}

\medskip
Fixing $s\in\re$, we choose $p\in\re$ such that $0<|p|<\Pi(s)$.  Using Proposition \ref{prop:open}, we can find $f\in C^\infty(\pd\Omega)$ such that the solution $u=u[f]$ to \eqref{eq} has the following properties:
\eqal{
f(0)=s,\quad \frac{\pd f}{\pd x^1}(0)=p,\quad \frac{\pd u}{\pd x^2}(0) =0,\quad \|f-s\|_{C^{2,\al}(\pd\Omega)}\le\Pi_1(s).
}
Since $\Gamma_a=\Gamma_{\ti a}$, we infer at $x=0$,
\eqal{
\ti a \frac{\pd\ti u}{\pd x^2}(0)=a\frac{\pd u}{\pd x^2}(0)=0,
}
and hence $\ti u_2(0)=0$ as well.  Substituting into \eqref{aleq1} yields
\eqal{
pa_{p_2}(s,p,0)=p\ti a_{p_2}(s,p,0).
} 
Substituting this and $u_2=\ti u_2=0$ into \eqref{deteq1} yields
\eqal{
(a+pa_{p_1})a|_{(s,p,0)}=(\ti a+p\ti a_{p_1})\ti a|_{(s,p,0)}.
}
Putting
\eqal{
\mc A(s,p):=a(s,p,0)^2-\ti a(s,p,0)^2,
}
we deduce
\eqal{
\mc A+\frac{1}{2}p\mc A_p=0,
}
and hence
\eqal{
\mc A(s,p)=c(s)|p|^{-1/2}.
}
Since $\mc A(s,p)$ is smooth, we conclude $c=0$:
\eqal{\label{mcA0n2}
\mc A(s,p)=0,\qquad s\in\re,\quad |p|<\Pi(s).
}
As in the proof for $n\ge 3$, since $O(2)\cdot e_1=S^1$, where $e_1=(1,0)$ and $O(2)$ is the group of orthogonal transformations on $\re^2$, invoking Euclidean isometries completes the proof.

\begin{rem}
\label{rem:n2}
As in Remark \ref{rem:n3}, Theorem \ref{thm:open} still holds for $n=2$ if we replace $a$ and $\ti a$ by
$$
F(x,s,p,b(s,p))\quad \text{ and }\quad F(x,s,p,\ti b(s,p)),
$$
where $t\mapsto F(x,s,p,t)$ is injective for each fixed $(x,s,p)$.  Indeed, from \eqref{mcA0n2}, we obtain   
$$
F(0,s,p,0,b(s,p,0))^2=F(0,s,p,0,\ti b(s,p,0))^2,\qquad s,p\in\re,\quad |p|<\Pi(s).
$$
By \eqref{coer}, we have $F\ge 1$, so the map $t\mapsto F(x,s,p,t)^2$ is also injective for fixed $(x,s,p)$, and we obtain
$$
b(s,p,0)=\ti b(s,p,0),\quad s,p\in\re,\quad |p|<\Pi(s).
$$
As before, the Euclidean invariance argument completes the proof.
\end{rem}

\section*{Acknowledgements}
I would like to thank my advisor Prof. Gunther Uhlmann for suggesting this problem, and Ryan Bushling for discussions.  R. Shankar was partially supported by the National Science Foundation Graduate Research Fellowship Program under grant No. DGE-1762114.

\end{document}